\documentclass[12pt]{amsart}
\usepackage{amssymb,latexsym,amsmath}
\usepackage{amsthm}
\usepackage{dsfont}
\usepackage{graphicx}
\usepackage{verbatim}
\usepackage{setspace}
\usepackage{fancyhdr}

\theoremstyle{plain}
\newtheorem{theorem}{Theorem}

\newtheorem*{main}{Main Theorem}

\newtheorem*{conjecture}{Conjecture}

\theoremstyle{definition}
\newtheorem*{definition}{Definition}

\theoremstyle{remark}

\newtheorem{remark}{Remark}

\begin{document}
\title{Crossing Changes in Closed 3-braid Diagrams}
\author{Chad Wiley}
\address{Emporia State University}
\email{cwiley1@emporia.edu}

\begin{abstract}
A crossing in a knot is nugatory if changing the crossing does not change the knot type.  Using an invariant of certain types of closed 3-braid diagrams, we show that if a closed 3-braid contains a nugatory crossing then its braid index is one or two.  This proves a special case of a conjecture on nugatory crossings due to Xiao-Song Lin.
\end{abstract}

\maketitle

\section{Introduction}\label{S:intro}

This paper is motivated by the following problem:  given a knot, is it possible that a single crossing change can produce a knot which is equivalent to the original?  As stated, this problem is easily solved.  We can readily produce examples of knots which have this property (see figure \ref{Fi:example}).

\begin{figure}[htp]
\centering\scalebox{.50}{\includegraphics{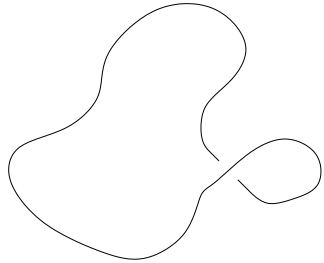}}
\caption{A knot with a crossing whose change does not affect the knot type.}
\label{Fi:example}
\end{figure}

A crossing in a knot which has this property is known as \emph{nugatory}.  Once we know that nugatory crossings exist, we naturally turn our attention toward classifying them.  Certainly a crossing introduced by taking a knot and giving one arc of it a single twist is nugatory.  More generally, a knot formed by taking two knots $K_1$ and $K_2$ and forming the ``twisted" connected sum contains a nugatory crossing (see figure \ref{Fi:connectsum}).

\begin{figure}[htp]
\centering\scalebox{.50}{\includegraphics{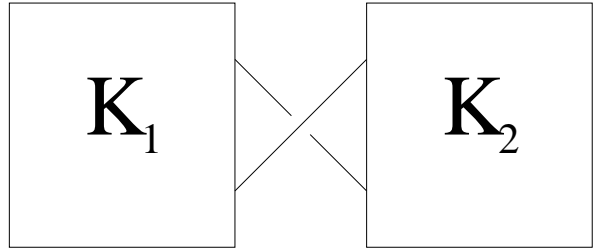}}
\caption{The ``twisted" connected sum of $K_1$ and $K_2$.  The crossing shown is nugatory.}\label{Fi:connectsum}
\end{figure}

\subsection{The nugatory crossing conjecture}

There is a conjecture, due to Xiao-Song Lin, which states that the crossing shown in figure \ref{Fi:connectsum} is the only kind of nugatory crossing (This is problem 1.58 in \cite{KirbyProblemList}).  In order to get a rigorous statement of this conjecture, we must first define exactly what we mean by a crossing change for a knot in $S^3$.  Suppose $K$ is an oriented knot or link in $S^3$, and suppose we have a disk $D$ which intersects $K$ at exactly two points with opposite orientation.  Such a disk is called a \emph{crossing disk} for $K$.  Suppose $D \times I$ is a bicollar of $D$ which intersects $K$ in a pair of arcs.  There is a homeomorphism of $D \times I$ that fixes $D \times \{0\}$ and gives $D \times \{1\}$ a full rotation. Replacing $D \times I$ with its image under such a homeomorphism is called a \emph{crossing change}.  In \cite{MartyCrossing} it is shown that this notion of a crossing change generalizes the notion of a crossing change in a specific diagram of a knot in the sense that a crossing change in a diagram determines a crossing change in $S^3$ and for any crossing change in $S^3$ one can find projections of the knot before and after the change that are identical except in a single crossing.

Using this definition of a crossing change, we can restate the conjecture as follows:

\begin{conjecture}[The nugatory crossing conjecture]
If a crossing change in a knot $K$ is nugatory, i.e. if the knots before and after the crossing change are equivalent, then the crossing disk $D$ bounds a disk in $S^3$ disjoint from $K$.
\end{conjecture}

The union of these two disks produces a twice-punctured sphere which decomposes $K$ into a connected sum, although one of the summands may be trivial.  This implies that there is a projection of the knot which looks like figure~\ref{Fi:connectsum}.

While this conjecture is still unproven, some partial results are known.  Probably the most important of these states that the conjecture is true for the unknot.  This follows from an argument of Scharlemann and Thompson, which uses the machinery of sutured manifold theory (see \cite{MartyCrossing} for a proof).  In \cite{TorisuNugatory}, Torisu proved that 2-bridge knots satisfy the conjecture.

The conjecture is also known to hold for fibred knots, due to Kalfagianni \cite{Fibred}.

\subsection{The Main Theorem}

The goal of this paper is to prove another special case of the conjecture.  By examining the behavior of closed braids in Hecke algebras, we can prove the following:

\begin{main}
Lin's conjecture holds for closed 3-braid diagrams.  Specifically, if a crossing in a closed 3-braid diagram is nugatory, then the braid index of the knot type it represents is 1 or 2. Thus, by the results of Scharlemann and Torisu, if a closed 3-braid diagram contains a nugatory crossing then when the closed braid is considered in $S^3$ the crossing disk corresponding to that crossing bounds a disk in $S^3$ disjoint from the closed braid.
\end{main}

The plan of the paper is as follows:  First, we discuss the behavior of closed braids in a specific algebraic setting, the Iwahori-Hecke algebra.  We describe an algorithm that inputs a closed braid and outputs a linear combination of the basis elements of the algebra.  This algorithm gives us a way to better understand the form that a closed braid takes in the algebra.  The full result requires the use of a powerful theorem of Birman and Menasco, known as the Markov Theorem without Stabilization, to generalize results about elements of the algebra back to closed braids.

\section{Definitions}\label{S:definitions}

Since we will be considering a conjecture about knots from the point of view of braid theory, it will be helpful to recall a few definitions and theorems which help connect the two ideas.

Given a braid $b\in B_n$, we denote the closure of the braid by $\hat{b}$.  The closure of a braid is typically drawn as a set of concentric circles with certain pairs of neighboring arcs removed and replaced with crossings.  We'll call the common center of the concentric circles the \emph{axis} of the closed braid diagram.

It is a theorem of Alexander that any braid diagram can be transformed into a diagram of a knot or link by forming its closure, and conversely any knot or link diagram can be cut apart to produce a braid.

This association of braids and links is non-unique.  In fact, there are many inequivalent braids which, when closed, produce equivalent knots. The problem of describing when this can occur was solved by Markov, and a pair of braids whose closures are equivalent knots are typically called \emph{Markov equivalent}.

\begin{theorem}{(Markov's Theorem)}\label{T:Markov}
Suppose that $b$ and $b'$ are braids which have the property that their closures $\widehat{b}$ and $\widehat{b'}$ are equivalent knots.  Then there exists a finite sequence $$b = b_1 \rightarrow \ldots \rightarrow b_k = b'$$ where each passage $b_i \rightarrow b_{i+1}$ is given by one of the following moves, known as \emph{Markov moves}.  In each move below, assume that $b$ is an $n$-braid and $\iota:  B_n \rightarrow B_{n+1}$ is the inclusion map.
\begin{itemize}
\item Conjugation:  $b \rightarrow aba^{-1}$ for some $a \in B_n$.
\item Stabilization:  $b \rightarrow \iota(b)\sigma_{n}$ or $b \rightarrow \iota(b)\sigma_{n}^{-1}$
\item Destabilization:  $\iota(b)\sigma_{n} \rightarrow b$ or $\iota(b)\sigma_{n}^{-1} \rightarrow b$
\end{itemize}
\end{theorem}

\section{Closed braids in Hecke algebras}\label{S:algorithm}

Our first goals are to define a useful quotient of the Iwahori-Hecke
algebra, define a basis for this quotient, and describe an algorithm
that allows us to describe any closed braid in terms of this basis.

\subsection{The Iwahori-Hecke algebra}

Let $R$ be a domain.  Fix units $A$ and $B$ in $R$ and fix an integer $n$.
We define the Iwahori-Hecke algebra $H_n$ to be the
associative $R$-algebra generated by $T_1, \ldots , T_{n-1}$ and
subject to the relations
\begin{itemize}
\item $T_i T_j = T_j T_i$ if $|i-j| > 1$
\item $T_i T_{i+1} T_i = T_{i+1} T_i T_{i+1}$
\item $T_i^2 - B T_i - A = 0$
\end{itemize}

There is a natural homomorphism from the braid group $B_n$ to $H_n$ given by $\sigma_i \mapsto T_i$,
where $\{\sigma_1, \ldots , \sigma_{n-1}\}$ are the standard generators
of the braid group $B_n$.  Thus $H_n$ is isomorphic to the group algebra $RB_n$ modulo the skein relation given in Figure \ref{Fi:skein}
(which is implied by the third relation above).

\begin{figure}[htp]
\centering\scalebox{.50}{\includegraphics{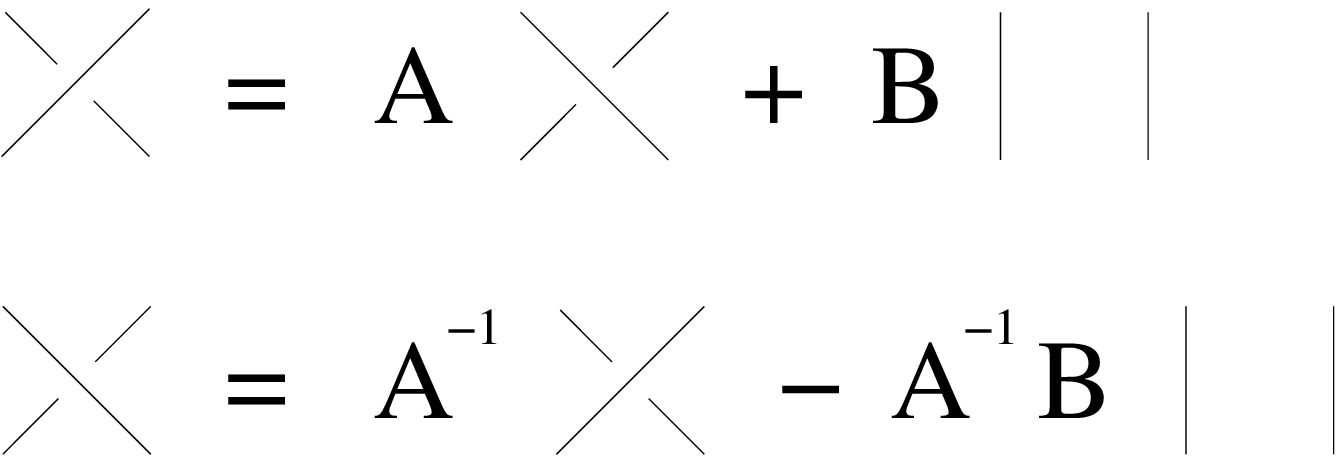}}
\caption{Two equivalent forms of the skein relation in $H_n$}\label{Fi:skein}
\end{figure}

The skein relation should be interpreted as follows:  suppose $b_+$, $b_-$, and $b_0$ are braid diagrams which are identical except at a single crossing, where $b_+$ has a positive crossing, $b_-$ has a negative crossing, and $b_0$ has no crossing.  Then the three diagrams are related in $H_n$ by the formula $b_+ = A b_- + B b_0$.

The construction of $H_n$ given here follows the construction in \cite{algorithm}.  The original construction was defined in \cite{Iwahori}.

$H_n$ is a useful setting for the study of braids, but since we are
primarily interested in knots and links, one more step must be taken.
Let $V_n$ be the quotient of $H_n$ by the extra relation $ab = ba$,
where $a$ and $b$ are elements of $H_n$.  Geometrically, this relation
has the effect of replacing any braid diagram with its closure.

$V_n$ can now be described as an $R$-module of linear combinations of
closed braid diagrams subject to the skein relations in figure~\ref{Fi:skein}.

Our first theorem gives some insight into the structure of $V_n$.
Recall that a \emph{partition} of a positive integer $n$ is a
sequence of positive integers $(n_1, \ldots , n_m)$ such that
$n_1 \geq \ldots \geq n_m \geq 0$ and $n_1 + \ldots + n_m = n$.

\begin{theorem}\label{T:dimension}
$V$ is a finite dimensional free $R$-module with dimension $P(n)$, where
$P(n)$ is the number of partitions of $n$.
\end{theorem}

We can define an explicit basis for the $R$-module $V_n$.  For a
positive integer $k$, define $b_{(k)}$ to be the $k$-braid
$\sigma_{k-1} \ldots \sigma_2 \sigma_1$.  Let
$\lambda = (\lambda_1, \ldots , \lambda_m)$ be a partition of $n$.
Define $v_\lambda$ to be the $n$-braid obtained
by placing the braids $b_{(\lambda_1)}, \ldots , b_{(\lambda_m)}$ side by
side from left to right and regarding the whole as a single braid (see figure \ref{Fi:basis}).
The set $\{\hat{v_\lambda}\}$, where $\hat{v_\lambda}$ is the closure of $v_\lambda$ and $\lambda$ ranges over the partitions of $n$, is a basis for $V_n$.

\begin{figure}[htp]
\centering\scalebox{1.0}{\includegraphics{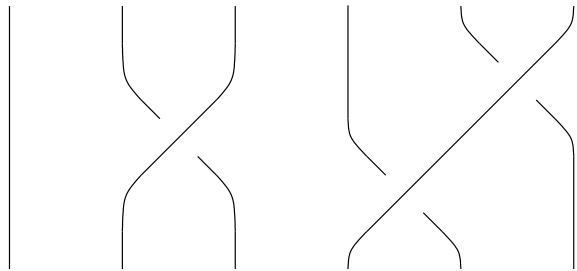}}
\caption{The 6-braid $v_{(1,2,3)}$}\label{Fi:basis}
\end{figure}

The proof that $\{\hat{v_\lambda}\}$ is a basis will require an algorithm
described by Bigelow in \cite{algorithm} and originally due to Hoste and Kidwell \cite{dichromatic}.  The purpose of this algorithm is to provide a systematic method for describing any closed braid in terms of this basis.  The proof of Theorem \ref{T:dimension} is provided in \cite{algorithm}.

\subsection{The algorithm}

The algorithm itself is straightforward.  Given a closed, oriented $n$-braid
diagram, choose a basepoint anywhere on the braid that is not a
crossing point.  Isotope the braid by
dragging a neighborhood of the basepoint toward the axis of the
braid diagram, pulling it over all the other strands, so that the basepoint is
at least as close to the axis as any other point.
Consider every crossing in the diagram to be unlabeled.
Starting at the basepoint, travel along the braid following its orientation.  When an unlabeled crossing is encountered, we label it
\emph{good} if we are traveling along the overcrossing strand and
\emph{bad} if we are traveling along the undercrossing strand.  See figure \ref{Fi:goodbad}.

\begin{figure}[htp]
\centering\scalebox{0.75}{\includegraphics{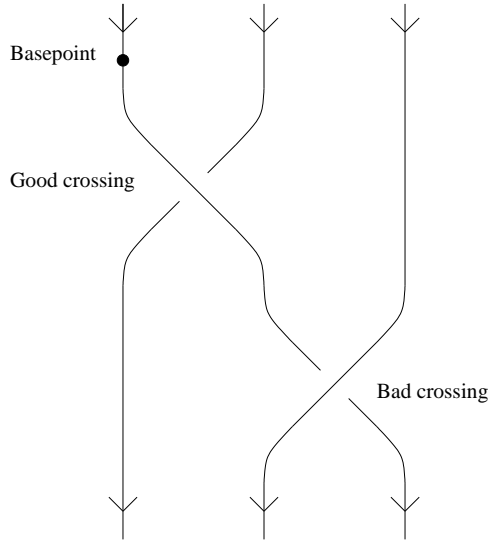}}
\caption{Labeling crossings good and bad}\label{Fi:goodbad}
\end{figure}

Whenever we encounter a bad crossing, we immediately use the skein
relation to change our diagram into a linear combination of braid
diagrams, one of which has replaced the bad crossing with a good
crossing, and one of which has replaced it with no crossing.  Using
the same basepoint, we then run the algorithm on each of the resulting
diagrams.

When we eventually travel back to the basepoint, we have obtained a
diagram in which the component of the link containing the basepoint
contains no bad crossings.  In particular, this means that it
contains only overcrossings with other components of the link.  As
such, we may isotope the entire component to the outer edge of the
braid, moving it over all the other components, so that it is the
farthest component of the link from the axis of the closed braid.  We
then apply the algorithm to any remaining components, choosing a
new basepoint as necessary.

Once a crossing is labeled (and possibly changed) for the first time in this process, it will never again need to be changed.  This is because a good crossing will always be labeled good in further iterations of the algorithm.  The number of unlabeled crossings decreases with each iteration, and so the process will terminate in a linear combination of closed braid
diagrams that contain no bad crossings.  It is not difficult to see
that the resulting diagrams must be braid isotopic to $\hat{v_\lambda}$ for
some partition $\lambda$ of $n$.

An example of this algorithm is given below in figure \ref{Fi:algorithm}.

\begin{figure}[htp]
\centering\scalebox{0.30}{\includegraphics{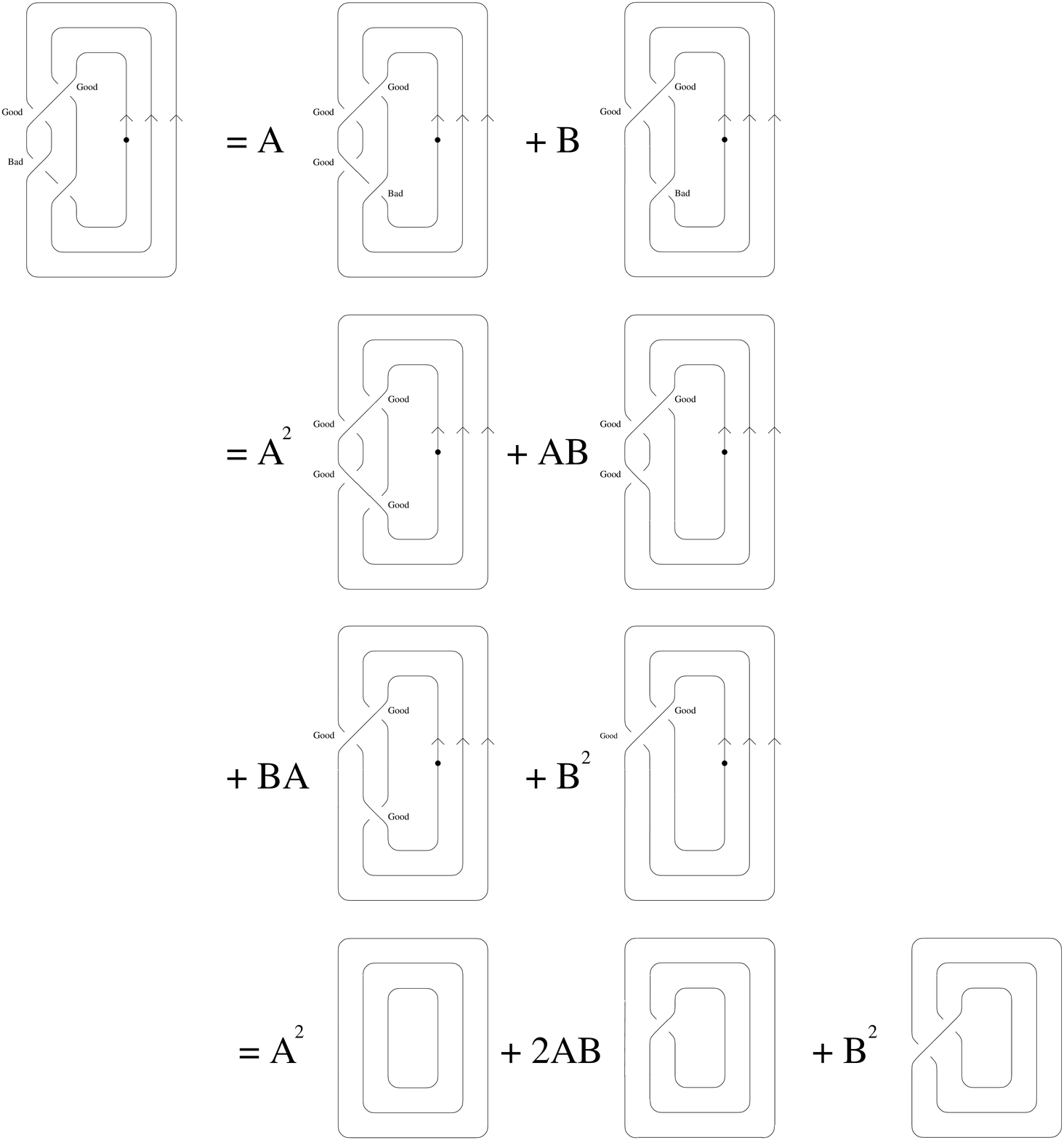}}
\caption{The algorithm in action}\label{Fi:algorithm}
\end{figure}

This algorithm will be useful to us in the sections to follow, so it is to our benefit to make some remarks about it.

\begin{remark}
The results of the algorithm can be described in terms of a binary tree with labeled edges.  Since the skein relation transforms one closed braid diagram into a linear combination of two other diagrams, each application of the skein relation can be described by a branch in the tree with edges labeled by the appropriate coefficients.  See figure \ref{Fi:trefoil}.  The final coefficients in the linear combination at the end of the algorithm can be recovered from the binary tree by examining all possible paths from the root vertex to the leaf vertices.
\end{remark}

\begin{figure}[htp]
\centering\scalebox{.50}{\includegraphics{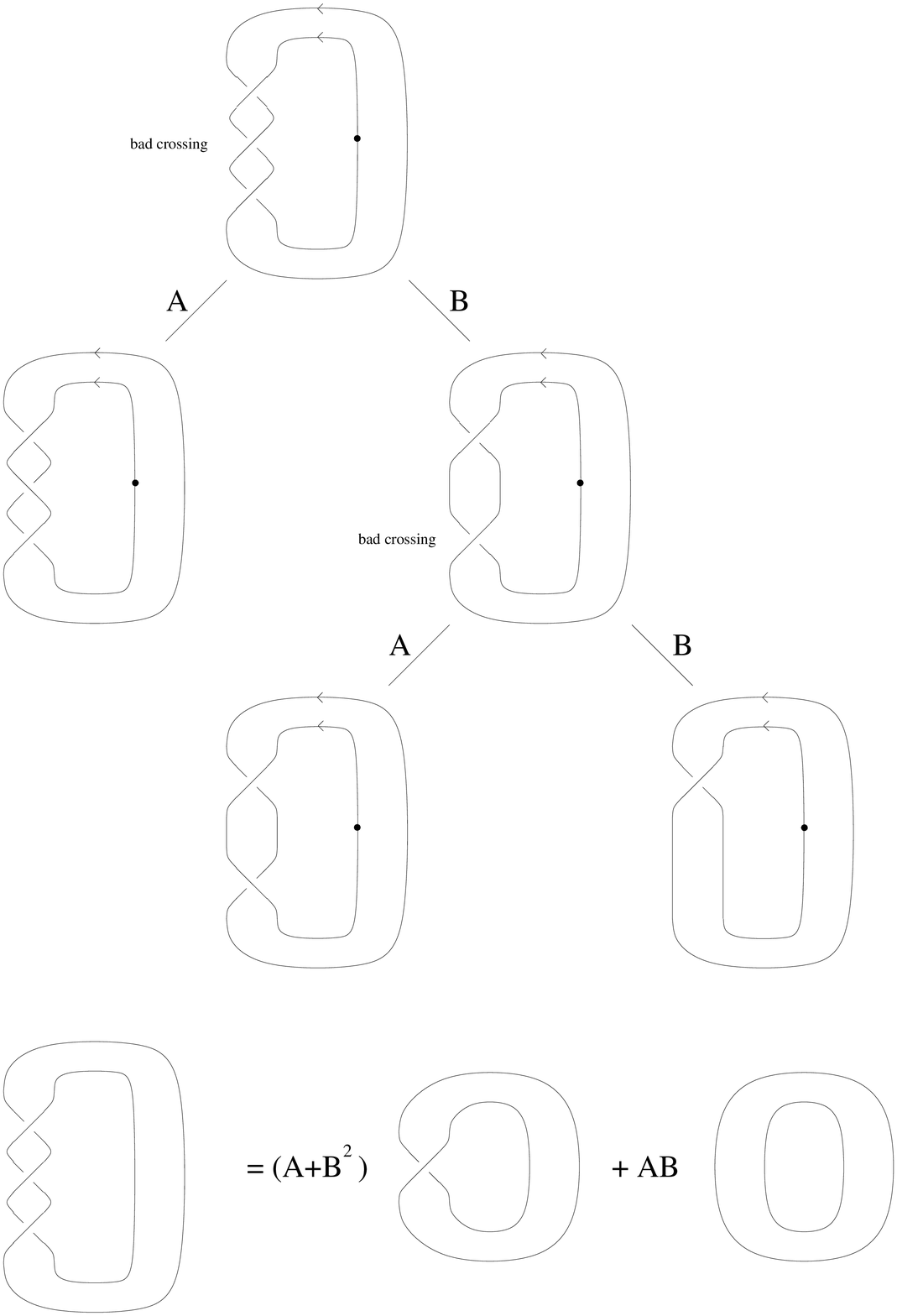}}
\caption{Applying the algorithm to the trefoil.}\label{Fi:trefoil}
\end{figure}

\begin{remark}\label{R:label}
It is possible to assign a good/bad label to every crossing in the diagram before beginning the algorithm, rather than assigning labels at every stage as described above.  Simply follow a strand, starting at the chosen basepoint, and assign labels at each crossing as above but take no action if a ``bad" label is assigned.  Continue this process until you come around to the basepoint again.  In the case of a diagram which represents a knot, this is all that is needed.  In the case of a link diagram, however, some crossings will remain unlabeled when the basepoint is reached.  In this case, we must choose a new basepoint on a different component of the link and repeat the previous step, labeling every unlabeled crossing we come across, until our new basepoint is reached again.  Continue choosing new basepoints and labeling crossings until all components of the link and all crossings of the diagram are accounted for.

In practical terms, this alternate method of assigning labels is highly inefficient.  This is due to the fact that every time the algorithm resolves a crossing, it creates a diagram in which that crossing is removed.  This may have the side effect of changing the labels on some of the other crossings in the diagram, in which case the entire labeling process would have to be repeated.  However, it is very useful to know that, in theory, each crossing in a particular closed braid diagram can be uniquely assigned a good/bad label (up to valid choices of basepoint).  We will make use of this fact in Theorem~\ref{T:parity}.
\end{remark}

\begin{remark}
It is important to notice that the linear combination that results from the algorithm is in general not a knot or link invariant.  A simple example is given in Figure \ref{Fi:notinvariant}.

\begin{figure}[htp]
\centering\scalebox{.50}{\includegraphics{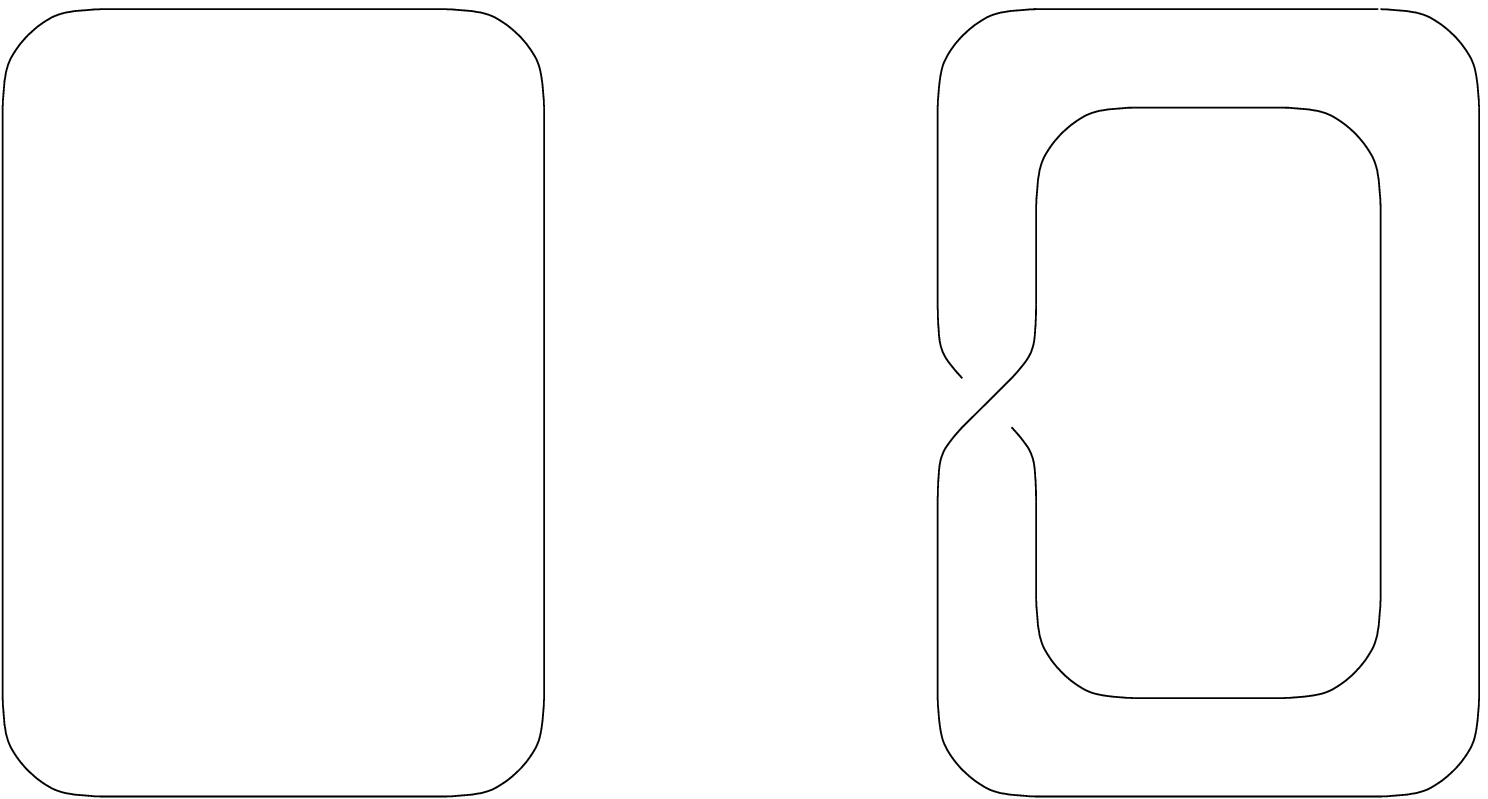}}
\caption{Two closed braids which share the same knot type but do not have the same output under the algorithm.}
\label{Fi:notinvariant}
\end{figure}

The obstruction involves Markov's Theorem (Theorem \ref{T:Markov}).  The three types of Markov moves are summarized below in the language of $V_n$.  Here $b$ is an $n$-braid and $\iota : H_n \to H_{n+1}$ is the inclusion map.

\begin{itemize}
\item braid isotopy
\item Stabilization:  $\hat{b} \rightarrow \widehat{\iota(b)\sigma_{n+1}}$ or $\hat{b} \rightarrow \widehat{\iota(b)\sigma_{n+1}^{-1}}$
\item Destabilization:  $\widehat{\iota(b)\sigma_{n+1}} \rightarrow \hat{b}$ or $\widehat{\iota(b)\sigma_{n+1}^{-1}} \rightarrow \hat{b}$
\end{itemize}

The algorithm is invariant under braid isotopy, but not stabilization or destabilization.  However, as shown in \cite{algorithm}, we can use trace functions to obtain the HOMFLY and Jones polynomials from the algorithm's output.
\end{remark}

\section{Using the algorithm to study crossing changes}\label{S:crossing change}

Since the algorithm is not an invariant of knots, it is not
immediately apparent how it might be used to study crossing changes
of knots.  This obstacle is overcome for a certain class of knots by
the Markov Theorem Without Stabilization (or MTWS) of Birman and
Menasco, which was proved in \cite{MTWS}.

\subsection{The MTWS}

The MTWS is the result of an attempt to better understand the inner workings of Markov's Theorem.  While Markov's theorem is is quite powerful and useful, especially when one is concerned with developing knot invariants using the machinery of braid theory, it does have a major drawback:  given two braids which are Markov equivalent, Markov's theorem does not supply an algorithm for changing one braid into the other using Markov moves.  It does not even give any indication of how many moves may be required.  For instance, it is possible that to change one 3-braid into another Markov equivalent 3-braid might require a sequence of a million consecutive stabilizations, a set of braid isotopies and conjugations, and a million consecutive destabilizations.

The MTWS does not directly provide an algorithm to transform one braid into another, but it does show that if such a transformation can be done, it can be done with a minimum of complexity.  It does this by defining a new set of moves which are similar in flavor to Markov's moves, but are required to to always be nonincreasing on braid index.  Thus the sequences of intermediate braids between one braid and another are monotone decreasing on braid index.  Together with a complexity function provided by the theorem, this creates an environment where we can study Markov equivalence between closed braids without worrying about the problem mentioned above (huge jumps in complexity).

\begin{theorem}{(The Markov Theorem Without Stabilization (MTWS))}\label{T:MTWS}
Let $X_+$ and $X_-$ be closed braid diagrams which represent the
same knot type, and suppose that the index of $X_-$ is equal
to the braid index of the associated knot type.  Then there exist the following:
\begin{enumerate}
    \item A complexity function with values in $\mathbb{Z}_+ \times \mathbb{Z}_+ \times \mathbb{Z}_+$ associated to the pair $(X_+,X_-)$.
    \item For each braid index $m$, a set $\mathfrak{T}(m)$ of templates, each of which determines a
        move (in the sense of Reidemeister moves or Markov moves) which is
        non-increasing on braid diagram index.
    \item A sequence $$X_- = X_-^1 \rightarrow \ldots \rightarrow X_-^p = X'_-$$ in which each passage $X_-^i \rightarrow X_-^{i+1}$ is strictly complexity-reducing and is achieved, up to braid isotopy, by an exchange move.
    \item A sequence $$X_+ = X_+^1 \rightarrow \ldots \rightarrow X_+^q = X'_+$$ in which each passage $X_+^i \rightarrow X_+^{i+1}$ is strictly complexity-reducing and is achieved, up to braid isotopy, by an exchange move or a destabilization.
    \item A sequence $$X'_+ = X^q \rightarrow \ldots \rightarrow X^r = X'_-$$ in which each passage $X^i \rightarrow X^{i+1}$ is strictly complexity-reducing and is achieved, up to braid isotopy, by an exchange move, a destabilization, and admissible flype, or a move defined by one of the templates in $\mathfrak{T}(l)$, where $l$ is the braid index of the diagram $X^i$.
\end{enumerate}
\end{theorem}

In order to understand the terminology of the MTWS, a few definitions are
in order.

A \emph{block-strand diagram} is a diagram as in figure \ref{Fi:exchange}.
When braids of the appropriate sizes are substituted into the blank
boxes (called a braiding assignment), it becomes a closed braid diagram.  Frequently, multiple parallel strands in a block-strand diagram are drawn as a single weighted strand, the number of intended strands being indicated next to the strand itself.  Figures \ref{Fi:exchange} and \ref{Fi:flype} include weighted strands to indicate that they can represent closed braids of any index.  A \emph{template}
is a pair of block-strand diagrams $D$ and $E$, both with blocks $B_1,\ldots,B_k$, together with an isotopy taking $D$ to $E$ in such a way that any fixed braiding assignment to the blocks yields closed braids $D'$ and $E'$ which represent the same knot type.  Examples
of templates are given below, which define the exchange and
admissible flype moves listed in the statement of the MTWS.  The
destabilization move listed above is the same as the move given in
Markov's Theorem.  All moves used in the MTWS preserve the knot type
of the diagrams involved.

\begin{figure}[htp]
\centering\scalebox{1.0}{\includegraphics{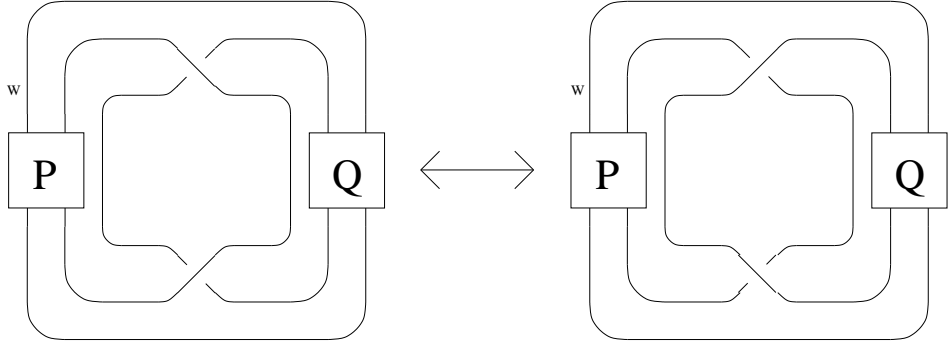}}
\caption{The exchange template.}\label{Fi:exchange}
\end{figure}

\begin{figure}[htp]
\centering\scalebox{1.0}{\includegraphics{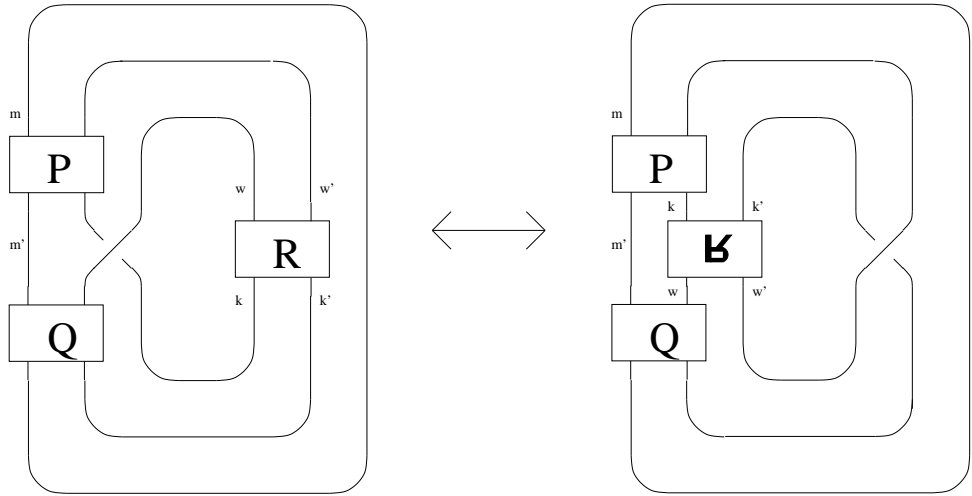}}
\caption{The flype template.}\label{Fi:flype}
\end{figure}

\begin{remark}
A flype move is called \emph{admissible} if $w'-k = k'-w \geq 0$.  This ensures that the move is non-increasing on braid index.  In the case of 3-braids, it is easy to see that all flypes are admissible.
\end{remark}

\begin{remark}
The full strength of the MTWS is not necessary for our purposes.  In particular, we have no use for the complexity function guaranteed by the theorem.  We are also working in the special case of closed 3-braids which represent knots of braid index 3, and this allows us to simplify the sequences in the MTWS to ones which only involve exchange moves since we need never perform a destabilization.  Since we are disregarding the complexity of the diagrams involved, there is no harm in reversing the sequence given in (3) and concatenating all three sequences into one.  This allows us to state the following abridged version of the MTWS which is more readily useful for us.

\begin{theorem}{(MTWS Lite)}\label{T:MTWS-Lite}
Let $X_+$ and $X_-$ be closed braid diagrams which represent the
same knot type, and suppose that the index of both diagrams is equal
to the braid index of the associated knot type.  Then for each braid index $m$ there
is a set $\mathfrak{T}(m)$ of templates, each of which determines a
move (in the sense of Reidemeister moves or Markov moves) which is
non-increasing on braid diagram index, and a sequence
$$X_+ = X^1 \rightarrow \ldots \rightarrow X^k = X_-$$
where each passage $X^i \rightarrow X^{i+1}$ is achieved, up to
braid isotopy, by an exchange move, an admissible
flype, or a move defined by one of the templates in
$\mathfrak{T}(m)$.
\end{theorem}

\end{remark}

\begin{remark}
It is known that in the cases $m = 1,2,$ and $3$, the sets $\mathfrak{T}(m)$ are empty.  (See \cite{MTWS})
\end{remark}

The MTWS gives us the tools necessary to prove the following key
theorem.

\subsection{Applying the algorithm to closed 3-braids}

\begin{theorem}\label{T:invariant}
Let $D$ and $E$ be closed 3-braid diagrams, and suppose that they
both represent the same knot type.  Suppose further that the braid index of their common knot type is 3.  Then the algorithm gives the same output for $D$ and $E$.
\end{theorem}

\begin{proof}
Begin by applying the MTWS to the diagrams $D$ and $E$.  This produces a finite sequence of closed 3-braid diagrams that begins with $D$ and ends with $E$.  Each diagram in the sequence differs from the last, up to braid isotopy, by an exchange move or an admissible flype move.  Since we know that the algorithm is invariant under braid isotopy, it remains only to show that it is invariant under 3-braid exchanges and flypes to prove the theorem.

Since exchange moves are only a special case of flypes in the 3-braid case, it suffices to show invariance under 3-braid flypes.  Refer to figure \ref{Fi:flypeinvar} below.

\begin{figure}[htp]
\centering\scalebox{.65}{\includegraphics{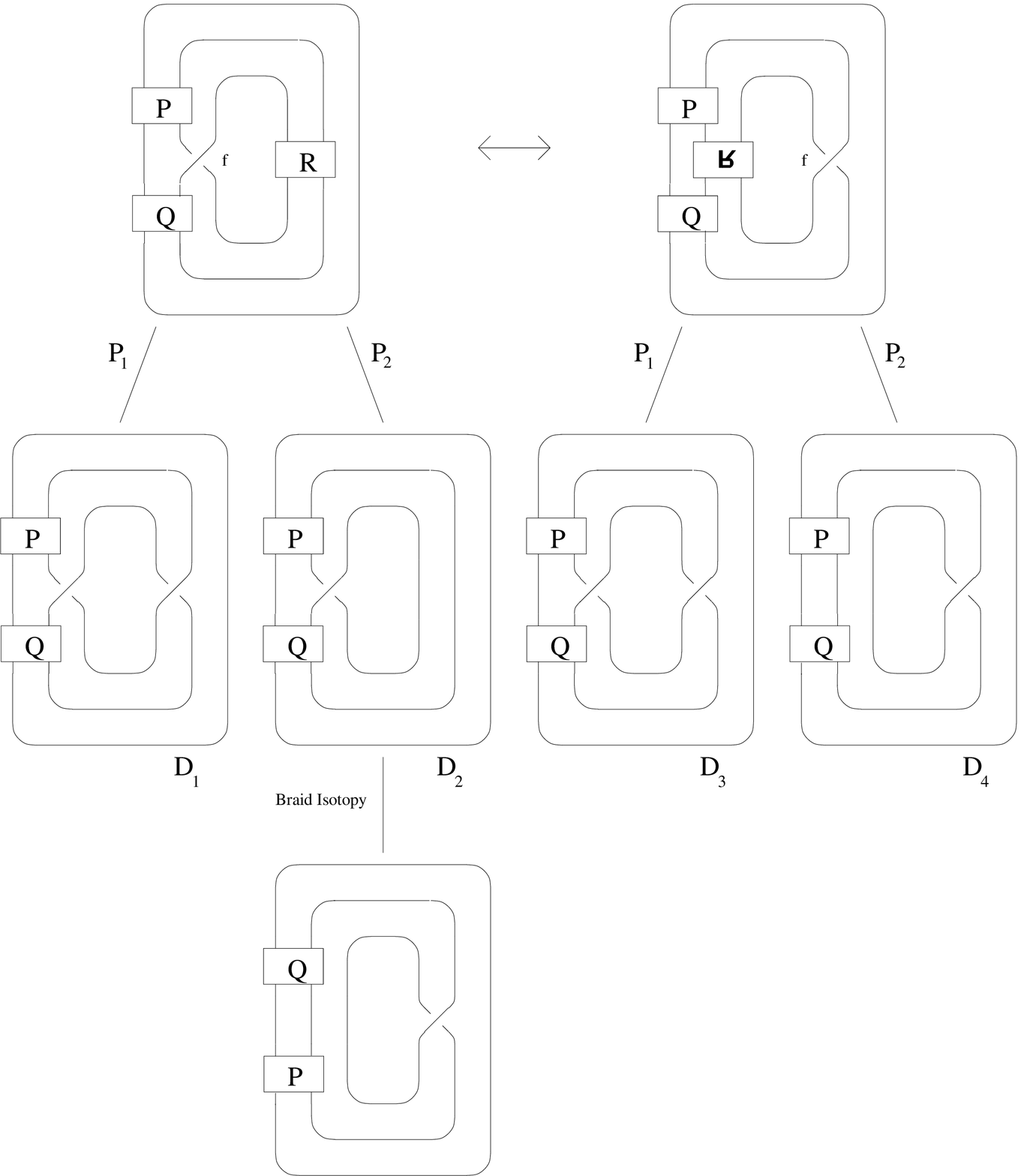}}
\caption{}\label{Fi:flypeinvar}
\end{figure}

We begin simplifying the diagram on the left by applying the skein relation to crossings in the 2-braids $R$ and \rotatebox[origin=c]{180}{\reflectbox{$R$}}.  Applying the skein relation to a crossing in a 2-braid results in a linear combination of a 2-braid with one fewer crossings and a 2-braid with two fewer crossings (assuming the original braid had at least two crossings).  Thus, repeated applications of the skein relation will eventually result in a linear combination of the trivial 2-braid and a 2-braid with one crossing.  This single crossing can be arranged to be either positive or negative by applying the skein relation to the single remaining crossing as necessary.  We follow this procedure with the 2-braid $R$, arranging that the single remaining crossing be of the same sign as the crossing marked $f$.  Since they are 2-braids, $R$ and \rotatebox[origin=c]{180}{\reflectbox{$R$}} are identical.  Thus repeating this process on \rotatebox[origin=c]{180}{\reflectbox{$R$}} will result in exactly the same coefficients in the resulting linear combination.  These coefficients are labeled $P_1$ and $P_2$ in figure \ref{Fi:flypeinvar}.  We label the four diagrams in the two resulting linear combinations $D_1$ through $D_4$.  The diagrams $D_1$ and $D_3$ are visibly identical.  $D_2$ and $D_4$ are also identical.  This follows from the fact that $P$ and $Q$ are 2-braids, and 2-braids are commutative.  Since all we have done is apply the skein relation and use braid isotopy, the algorithm must give the same result for both sides of the 3-braid flype template.  This completes the proof of the theorem.
\end{proof}

Part of the reason this theorem is interesting is because it gives a new invariant of a subset of closed 3-braids, namely the output of the algorithm.  Since one can apply trace functions to the algorithm's output to get the HOMFLY polynomial, this invariant is potentially more sensitive than the HOMFLY and Jones polynomials.  However, this invariant does not produce any new information about the knots with 10 or fewer crossings listed in \cite{Rolfsen} which have braid index 3.

\subsection{A theorem concerning bad crossings}

\begin{theorem}\label{T:parity}
Let $D$ and $E$ be closed braid diagrams.  Label each crossing in
each diagram as \emph{good} or \emph{bad} as in Remark \ref{R:label} above.  If $D$ and
$E$ have the same linear combination associated to them, then the
difference in the number of bad crossings between $D$ and $E$ is
even.
\end{theorem}

\begin{proof}
We claim that each term in each coefficient in the linear
combination has at least one power of $B$ in it with the exception
of exactly one, which looks like $A^k$.  We allow the possibility
that $k$ is zero.

To see why this is so, notice that the algorithm is merely a
systematic way of using the skein relation to change every bad
crossing in the diagram into a linear combination of a diagrams with
good crossings or no crossings.  When doing so, changing the
crossing to the opposite crossing introduces an $A$ or $A^{-1}$
term, whereas eliminating the crossing results in either a $B$ or
$-BA^{-1}$ term.  Using the binary tree formulation of the
algorithm, there is exactly one leaf arrived at by following only
edges marked $A$ or $A^{-1}$.  The other leaves must have some power
of $B$ in their associated term, since once a $B$ is introduced it
cannot be canceled.  The $k$ in the resulting $A^k$ term represents
exactly the difference between the number of positive bad crossings
and the number of negative bad crossings in the original diagram.

Let $p_D$ be the number of positive bad crossings in the diagram $D$, and let $n_D$ be the number of negative bad crossings in $D$.  Then $b_D = p_D + n_D$ is the total number of bad crossings in $D$.  Define $p_E$, $n_E$, and $b_E$ analogously for the diagram $E$.

By assumption, $D$ and $E$ have the same linear combination
associated to them.  We know that there is a unique $A^k$ term in each
linear combination obtained from the algorithm.  If we label the exponent of $A$ in $D$ and $E$ as $k_D$ and $k_E$ respectively, then we have the following:

\begin{align}
k_D &= k_E \notag \\
p_D - n_D &= p_E - n_E \notag \\
p_D + n_D &\equiv p_E + n_E \mod 2 \notag \\
b_D &\equiv b_E \mod 2 \notag
\end{align}

Thus the difference in the number of bad crossings between $D$ and
$E$ must be even.
\end{proof}

\section{Proof of the main theorem}\label{S:proof}
Armed with Theorem \ref{T:invariant} and Theorem \ref{T:parity}, the
proof of the main theorem is straightforward.

Let $D$ be a diagram of an arbitrary closed 3-braid.  The braid index of the knot represented by $D$
can be either 1, 2, or 3.

If the braid index of $D$ is 1, then $D$ is a diagram of the unknot.
The nugatory crossing conjecture is known to be true for the unknot,
following from an argument of Scharlemann and Thompson.
The case of braid index 2 is also known, since the nugatory crossing
conjecture is known to be true for 2-bridge knots and knots of braid
index 2 have bridge number 2.  Thus we need only concern ourselves
with the case where the braid index of $D$ is 3.

The idea of the proof for braid index 3 is to show that no crossing
in the diagram can actually be nugatory, since the result follows
from this immediately.  For the rest of the proof, we will tacitly
assume that the braid index of any diagram in question is 3. Under
this assumption, Theorem \ref{T:invariant} states that for closed 3-braids the
associated linear combination of basis diagrams is actually an invariant. Specifically, if two closed 3-braids have different
linear combinations associated to them, then they must represent
different knot types.  So if $D'$ is a diagram obtained from $D$ by
changing a single crossing, then we need only demonstrate that $D$
and $D'$ have different associated linear combinations.  But this
follows precisely from Theorem \ref{T:parity}, since a single crossing change either adds or subtracts one from the total number of bad crossings in the diagram (depending on choice of orientation and basepoint).  Thus $D$ and $D'$ do not represent the same knot type, and so the crossing in question could not have been nugatory.

\section{Conclusions and future work}\label{S:future}

The main theorem implies that if a closed 3-braid diagram contains a nugatory crossing, then the braid index of the knot it represents cannot be 3, since such diagrams have no nugatory crossings.  The braid index can therefore only be 1 or 2.  It is perhaps not surprising that such diagrams contain nugatory crossings, since by Markov's Theorem we then have a sequence
$$\text{1- or 2-braid diagram } = b_0 \rightarrow \ldots \rightarrow b_n = \text{ 3-braid diagram}$$
which must necessarily contain a stabilization, and the crossing introduced by a stabilization is nugatory.  It is easy to see that if one takes a crossing disk for the crossing introduced by a stabilization, it can be extended to a twice-punctured sphere which encapsulates the extra loop introduced by the stabilization.

This explanation does not, in itself, prove the conjecture for closed 3-braid diagrams because we would somehow have to prove that no crossings in the diagram are nugatory except for those which can be destabilized.  For diagrams of closed 3-braids which are explicitly drawn as stabilized 1- or 2-braids, this is straightforward.  But for in the general case of a closed 3-braid diagram representing a knot of braid index 1 or 2, it is a difficult problem.  A good possibility for future work would be to solve this problem so that the methods employed in this paper could provide a proof for all closed 3-braids.  Until then, the results of Scharlemann and Torisu suffice to fill this gap.

The main theorem does not use the full strength of Theorem \ref{T:parity}.  Theorem \ref{T:parity} states that any diagrams which have the same output under the algorithm must differ in the number of bad crossings by an even number.  In particular, if we start with a closed braid diagram (with any number of strands), then changing an odd number of distinct crossings in the diagram must result in a new diagram which does not have the same output under the algorithm.  But the algorithm is an invariant of closed 3-braids which cannot be represented as closed 1- or 2-braids.  So we can conclude the following:

\begin{theorem}\label{T:odd}
      If $D$ is a closed 3-braid diagram and the knot represented by $D$ has braid index 3, then any odd number of distinct crossing changes in the diagram must change the knot type.
\end{theorem}
This certainly need not be true of an even number of crossing changes, as evidenced by diagrams related by exchange moves.

Theorem \ref{T:odd} unfortunately does not generalize to knots in $S^3$, and this is the reason we must take such care with its hypotheses.  There is a much wider variety of possible crossing changes available for a knot in $S^3$ than for a single diagram of a knot.  An would be to take a knot in $S^3$ which has braid index 3 and simply introduce a nugatory crossing via a type I Reidemeister move.  Such a move will not affect the braid index of the knot since it does not change the knot type.  But now there is a single crossing which, when changed, will not change the knot type.

A more instructive example involves a theorem of Hirasawa and Uchida \cite{HUGordian}.  To state it properly, we need to define Gordian distance and the Gordian complex of knots.

\begin{definition}
Given two knots $K$ and $L$, we define the Gordian distance $d_G(K,L)$ between them to be the minimum number of crossing changes required to change $K$ into $L$.
\end{definition}

Since any knot can be unknotted by some sequence of crossing changes, the Gordian distance is always defined for any pair of knots.  It is easily seen that $d_G$ is a metric.

\begin{definition}
The Gordian complex of knots is a simplicial complex $\mathcal{G}$ whose vertex set is the set of isotopy classes of knots in $S^3$.  In addition, $n+1$ vertices in $\mathcal{G}$ form an $n$-simplex if the Gordian distance between any two of the vertices is 1.
\end{definition}

Several interesting results are known about the Gordian complex.  For instance, the fact that $d_G(K,L)$ is defined for any two knots $K$ and $L$ implies that $\mathcal{G}$ is connected.  In \cite{BaaderGordian}, Baader proved that for any pair of vertices in $\mathcal{G}$ which have a Gordian distance of 2, there are an infinite number of distinct paths of length 2 between them.

Hirasawa and Uchida proved the following theorem.

\begin{theorem}\label{T:Gordian}
Every vertex in $\mathcal{G}$ is contained in an infinite dimensional simplex of $\mathcal{G}$.  In other words, given any knot $K_0$, there is an infinite family of knots $K_0,K_1,\ldots$ with the property that $d_G(K_i,K_j) = 1$ for distinct $i$ and $j$.
\end{theorem}

Theorem \ref{T:Gordian} implies that an odd number of distinct crossing changes in a knot can indeed produce the same knot, if we consider the crossing changes in $S^3$ rather than in a fixed diagram of the knot.  This stands in contrast to Theorem \ref{T:odd}, implying that not everything which can be achieved by crossing changes in $S^3$ can be achieved within a single projection of a knot.  A possibility for future work would be to investigate the extent of this difference.

Unfortunately, the invariant presented here does not extend easily to closed 4-braids and beyond.  There are examples of closed 4-braids which differ by an exchange move for which the algorithm gives distinct results.  See figure \ref{Fi:nonexample}.  However, it is possible that the extension could still hold in the case of knots (rather than links, as is the case in figure \ref{Fi:nonexample}).

\begin{figure}[h]
\centering\scalebox{.50}{\includegraphics{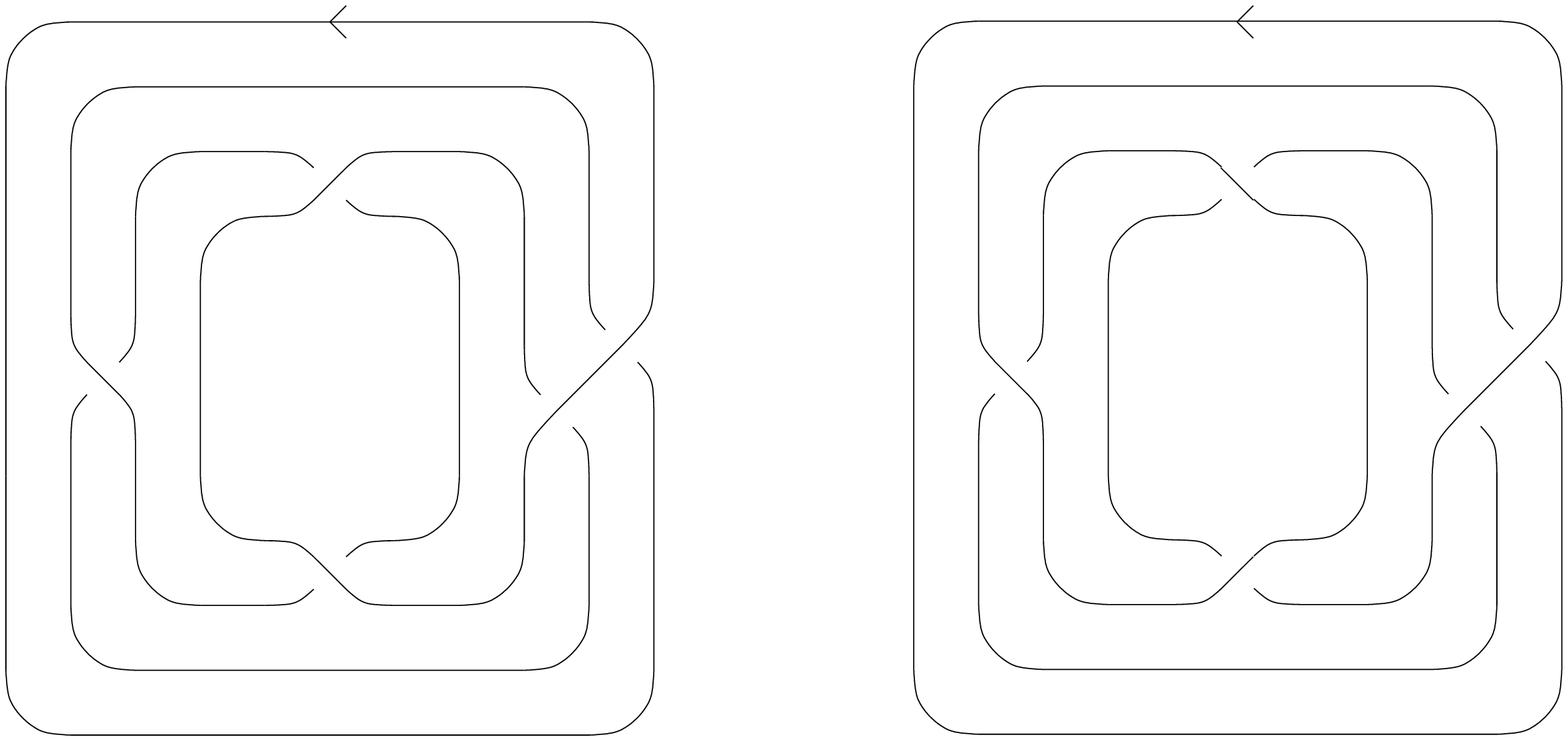}}
\caption{4-braids which are exchange equivalent but do not have the same output under the algorithm}\label{Fi:nonexample}
\end{figure}

Another problem is the fact that, as of this writing, there is no clear tabulation of the templates in the set $\mathfrak{T}(4)$.  An analog of Theorem \ref{T:invariant} is impossible without knowing exactly which moves are needed in the MTWS for 4-braids.  Birman and Menasco remark in \cite{MTWS} that there should be no real difficulty in determining the contents of $\mathfrak{T}(4)$, so this may be only a matter of time.

\nocite{BirmanBraid}
\nocite{MTWS2}

\newpage
\begin{singlespace}
\bibliographystyle{plain}
\bibliography{ChadBib}
\end{singlespace}
\end{document}